\documentclass[10pt]{amsart}
\usepackage{amsmath, amsthm, amsfonts, amssymb, hyperref,enumerate}
\theoremstyle{plain}
\newtheorem{thrm}{Theorem}[section]
\newtheorem{lmm}[thrm]{Lemma}
\newtheorem{prpstn}[thrm]{Proposition}

\newtheorem*{rmk}{Remark}
\numberwithin{equation}{section}
\DeclareMathOperator*{\Res}{Res}
\newcommand{\Mod}[1]{\ (\mathrm{mod}\ #1)}
\renewcommand{\mod}[1]{\mathrm{mod}\ #1}
\renewcommand{\le}{\leqslant}
\begin{document}
\title{Exponential sums over primes}

\author{James Maynard}
\address{Mathematical Institute, Radcliffe Observatory Quarter, Woodstock Road, Oxford OX2 6GG, England} 
\email{james.alexander.maynard@gmail.com}  

\author{Mayank Pandey}
\address{Courant Institute of Mathematical Sciences, 251 Mercer Street, New York, NY, 10012} 
\email{mayankpandey9973@gmail.com} 

\author{Maksym Radziwi\l\l} 
\address{Courant Institute of Mathematical Sciences, 251 Mercer Street, New York, NY, 10012} 
\email{maksym.radziwill@gmail.com}  

\thanks{The first author is supported by the European Research Council (ERC) under the European
Union’s Horizon programme (grant agreement No 101230920). The third author acknowledges support of NSF grant DMS-2401106}

\begin{abstract}
Let $\alpha=a/q+\epsilon$ with $(a,q)=1$, $q\le N^{1/2}$ and $|\epsilon|\le 1/(qN^{1/2})$, and let $B:=\max(q,qN|\epsilon|)$. We show that
\[
\Bigl|\sum_{n<N}\Lambda(n)e(n\alpha)\Bigr|\le N^{o(1)}\Bigl(\frac{N}{B^{1/2}}+N^{19/24}\Bigr).
\]
This improves on the classical bound of Vinogradov from 1937, which has $N^{4/5}$ in place of $N^{19/24}$.
\end{abstract}

\maketitle

\section{Introduction}

Exponential sums over primes
\[
S_N(\alpha):=\sum_{n<N}\Lambda(n)e(n\alpha)
\]
are a central object in additive prime number theory. Following Vinogradov's resolution of the ternary Goldbach problem for all sufficiently large odd integers \cite{Vinogradov}, pointwise bounds for $S_N(\alpha)$ have been a key input to applications of the Hardy--Littlewood circle method involving primes. They also underpin results on the distribution of $\alpha p$ modulo one and on Diophantine approximation with primes.

By Dirichlet's theorem on rational approximation, any $\alpha\in\mathbb{R}$ may be written as $\alpha=a/q+\epsilon$ with $(a,q)=1$, $q\le N^{1/2}$ and $|\epsilon|\le 1/(qN^{1/2})$, and the strength of the available bounds for $S_N(\alpha)$ is governed by the quantity
\[
B:=\max(q,qN|\epsilon|)\le N^{1/2}.
\]
Vinogradov's method, in the streamlined form given by Vaughan's identity \cite{Vaughan}, shows that
\begin{equation}
S_N(\alpha)\ll N^{o(1)}\Bigl(\frac{N}{B^{1/2}}+N^{4/5}\Bigr)
\label{eq:Vinogradov}
\end{equation}
The exponent $4/5$ has long stood as the limit of the classical methods for generic $\alpha$. Our main result improves it.

\begin{thrm}\label{thrm:MainThm}
Let $N \geq 1$.
Let $\alpha=a/q+\epsilon$ with $(a,q)=1$, $q\le N^{1/2}$ and $|\epsilon|\le 1/(qN^{1/2})$. Let $B:=\max(q,qN|\epsilon|)$. Then we have
\[
\Big |\sum_{n<N}\Lambda(n)e(n\alpha) \Big |\le N^{o(1)}\Bigl(\frac{N}{B^{1/2}}+N^{19/24}\Bigr).
\]
\end{thrm}
Theorem \ref{thrm:MainThm} should be compared with \eqref{eq:Vinogradov}: it replaces the term $N^{4/5}$ with $N^{19/24}$ (note that $4/5=19/24+1/120$). The most important range for our result is when $q\approx N^{1/2}$, which happens for `generic' $\alpha$. For $q$ smaller than $N^{1/2}$, Kumchev \cite{Kumchev} obtained a bound with $B^{1/2}N^{11/20}$ in place of $N^{4/5}$. This is stronger than the bound of Theorem \ref{thrm:MainThm} whenever $B \leq N^{1/2 - \delta}$ with $\delta > 1/60$ (the terms $B^{1/2}N^{11/20}$ and $N^{19/24}$ cross at $B=N^{29/60}$), but such $\alpha$ form a set of measure shrinking to zero as $N \rightarrow \infty$. Moreover, there is a large literature obtaining bounds of the form
\[
\ll \frac{N}{B^{1/2}}
\]
for $B \leq N^{\delta}$ with small $\delta > 0$.

Morally one can think of the $N/B^{1/2}$ term as coming from possible exceptional zeros with $\Re(s)$ close to 1; this term can be improved for all $q$ outside of a thin exceptional set. Under GRH we have the stronger estimate
\[
\sum_{n<N}\Lambda(n)e(n\alpha)=\frac{\mu(q)}{\phi(q)}\sum_{n<N}e(n\epsilon)+O(N^{1/2+o(1)}B^{1/2}),
\]
which gives an $O(N^{3/4})$ error term in the case of $q\approx N^{1/2}$.

Our proof combines the Heath-Brown identity with mean value theorems and fourth moment estimates for Dirichlet polynomials and $L$-functions, together with recent large values estimates for Dirichlet polynomials in the $q$-aspect \cite{Chen}, extending work of Guth and the first author \cite{GuthMaynard}.
As sketched in Section \ref{sec:Outline}, inserting these large values estimates into the standard zero density framework already improves the exponent $4/5$ to $67/84$; the stronger exponent $19/24$ of Theorem \ref{thrm:MainThm} comes from working with the Dirichlet polynomials directly, exploiting the fact that the critical factorizations behind the zero density argument cannot actually occur.

The paper is organized as follows. In Section \ref{sec:Outline} we outline the argument. In Section \ref{sec:Proposition} we prove our main technical result, Proposition \ref{prpstn:Dyadic}, which bounds the exponential sum over a dyadic range $n\sim x$ when $B\in[x^{4/9},x^{1/2}]$. In Section \ref{sec:Deduction} we deduce Theorem \ref{thrm:MainThm} from Proposition \ref{prpstn:Dyadic} via a dyadic decomposition, using an estimate of Kumchev \cite{Kumchev} for the ranges where $B$ is small.

\subsection*{Notation}
We write $e(t):=e^{2\pi it}$, and $\Lambda$ denotes the von Mangoldt function. We write $n\sim x$ to mean $x<n\le 2x$. We use the Vinogradov and Landau notation $f\ll g$, $f=O(g)$ to mean $|f|\le Cg$ for some constant $C>0$, and $f\asymp g$ to mean $g\ll f\ll g$; dependence of the implied constant on a parameter is indicated by a subscript. We write $A\lessapprox B$ (equivalently $B\gtrapprox A$) to mean $A\ll x^{o(1)}B$, where $o(1)$ denotes a quantity tending to $0$ as $x\to\infty$; all such statements are asserted for $x$ sufficiently large. Sums $\sum_{\chi\Mod{q}}$ are over all Dirichlet characters $\chi$ modulo $q$, $\tau(\chi):=\sum_{b\Mod{q}}\chi(b)e(b/q)$ denotes the Gauss sum, and $\phi$ is the Euler totient function.

\section{Outline}\label{sec:Outline}

As a proof-of-concept, let us sketch how to obtain an improvement over Kumchev's bound for a smoothed version of the sum in Theorem \ref{thrm:MainThm}. Let $w$ be a smooth function, compactly supported in $(1/2, 5/2)$. Splitting the sum into arithmetic progressions $\mod{q}$ via Dirichlet characters and then using Mellin inversion gives 
\begin{align*}
\sum_{n}w\Bigl(\frac{n}{x}\Bigr)\Lambda(n)e(n\alpha)&=\frac{1}{\phi(q)}\sum_{\chi\Mod{q}}\sum_{b\Mod{q}}\chi(b)e\Bigl(\frac{ab}{q}\Bigr)\sum_n \Lambda(n)\overline{\chi}(n)w\Bigl(\frac{n}{x}\Bigr)e(n\epsilon)\\
&=-\sum_{\chi\Mod{q}}\frac{\tau(\chi)\overline{\chi}(a)}{2\pi  i\phi(q)}\int_{c-i\infty}^{c+i\infty}W_{x,\epsilon}(s)\frac{L'}{L}(s,\overline{\chi})ds,
\end{align*}
where $\tau(\chi)$ is the Gauss sum, and $W_{x,\epsilon}(s)$ is the Mellin transform of $e(n\epsilon)w(n/x)$. Shifting contours (and ignoring the residual contour), we find
\begin{align*}
\sum_{n}w\Bigl(\frac{n}{x}\Bigr)\Lambda(n)e(n\alpha)&\approx\frac{W_{x,\epsilon}(1)}{\phi(q)}+O\Bigl(\sum_{\chi\Mod{q}}\frac{1}{q^{1/2}}\sum_{L(\rho,\overline{\chi})=0}|W_{x,\epsilon}(\rho)|\Bigr).
\end{align*}
It turns out that $W_{x,\epsilon}(\rho)$ is small unless $|\Im{\rho}|\asymp T:=\max(1,x|\epsilon|)$, in which case $|W_{x,\epsilon}(\rho)|\approx x^{\Re(\rho)}/T^{1/2}$, while $W_{x,\epsilon}(1)\ll x/T$. This gives
\begin{align*}
\sum_{n}w\Bigl(\frac{n}{x}\Bigr)\Lambda(n)e(n\alpha)\lessapprox \frac{x}{qT}+\sup_{\sigma} \frac{x^\sigma \sum_{\chi\Mod{q}} N(\sigma,T,\chi)}{(qT)^{1/2}}.
\end{align*}
Using the new zero density bound $\sum_\chi N(\sigma,T,\chi)\ll (qT)^{7(1-\sigma)/3+o(1)}$ in the range $\sigma\ge 5/7$, together with the classical Ingham-type zero density bound \cite[Chapter 12]{Montgomery}, $\sum_{\chi}N(\sigma,T,\chi)\ll (qT)^{3(1-\sigma)/(2-\sigma)+o(1)}$, in the range $1/2\le \sigma\le 5/7$ (the two bounds crossing at $\sigma=5/7$), would then give the bound
\[
\sum_{n}w\Bigl(\frac{n}{x}\Bigr)\Lambda(n)e(n\alpha)\lessapprox \frac{x}{(q T)^{1/2}}+x^{1/2} q^{1/2}T^{1/2}+x^{5/7}(qT)^{1/6}.
\]
Since $qT\ll x^{1/2}$ this gives the bound $x/(qT)^{1/2}+x^{67/84}$. The exponent $67/84$ gives a small improvement over the longstanding classical bound of $4/5$.

It turns out that the critical length $L$ of Dirichlet polynomials behind the zero density estimates for $\sum_{\chi} N(\sigma,T,\chi)$ in the argument given above is 
\[
L=q^{7/18}.
\]
When $q\approx x^{1/2}$ (which is the critical case for the exponent $67/84$), it is impossible to have all Dirichlet polynomials having combined length $x$ but individual lengths $L$. Therefore, by working with the Dirichlet polynomials directly we are able to do a little bit better, ultimately obtaining the bound $x^{19/24}$, as in Theorem \ref{thrm:MainThm}.

\section{Proof of Proposition \ref{prpstn:Dyadic}}\label{sec:Proposition}

\begin{lmm}\label{lmm:Saddle}
Let $\epsilon\ge 0$, and let $w$ be a smooth function supported on $[1/2,5/2]$ satisfying $\|w^{(j)}\|_\infty\ll_{j}1$ for all $j\in\mathbb{Z}_{\ge 0}$. Let
\[
W_{x,\epsilon}(s):=\int_0^\infty w\Bigl(\frac{u}{x}\Bigr)e(u\epsilon)u^{s-1}du.
\]
Then we have
\begin{itemize}
\item (Bound for large $t$): For any integer $j\ge 1$
\[
W_{x,\epsilon}(\sigma+it)\ll_{j,\sigma}x^\sigma \Bigl(\frac{1+x\epsilon}{|t|}\Bigr)^j.
\] 
\item (Bound for small $t$): If $\epsilon>0$ then for any integer $j\ge 1$
\[
W_{x,\epsilon}(\sigma+it)\ll_{j,\sigma}x^\sigma \Bigl(\frac{1+|t|}{x\epsilon}\Bigr)^j.
\] 
\item (Bound for $t$ not close to saddle point): If $t\le -10\pi x\epsilon$ or if $t\ge -\frac{\pi}{2} x\epsilon$ then
\[
W_{x,\epsilon}(\sigma+it)\ll_\sigma \frac{x^\sigma}{1+x\epsilon}.
\]
\item (Bound for saddle point): If $-10\pi x\epsilon\le t \le -\frac{\pi}{2} x\epsilon$ then
\[
W_{x,\epsilon}(\sigma+it)\ll_\sigma \frac{x^\sigma}{(1+x\epsilon)^{1/2}}.
\]
\end{itemize}
\end{lmm}
\begin{proof}
This follows from standard estimates. First note by a change of variables and the support of $w$, for $s=\sigma+it$
\[
W_{x,\epsilon}(s)=x^{\sigma+it}\int_{1/2}^{5/2} w(u)e(u x\epsilon)u^{s-1}du.
\]
The first claim follows now from integrating by parts $j$ times (and recalling that $w(u)\ll 1$ is supported on $u\in [1/2,5/2]$) where we integrate the $u^s$ factors and differentiate $w(u)u^{-1}e(x\epsilon u)$ factors. The second claim follows similarly from integrating the $e(u x\epsilon)$ factors and differentiating the $u^{s-1} w(u)$ factors. 

Combining the oscillatory factors, we see that
\[
W_{x,\epsilon}(\sigma+it)=x^{\sigma+it}\int_{1/2}^{5/2} w(u)u^{\sigma-1}e\Bigl(u x\epsilon+\frac{t}{2\pi}\log{u}\Bigr)du.
\]
Let $g(u):=u x\epsilon+t(2\pi)^{-1}\log{u}$, so that $g'(u)=x\epsilon+t/(2\pi u)$ is monotonic in $u$. If $t\le -10\pi x\epsilon$ then for all $u\in [1/2,5/2]$ we have
\[
g'(u)\le x\epsilon+\frac{t}{5\pi}\le -x\epsilon,
\]
while if $t\ge -\frac{\pi}{2} x\epsilon$ then for all $u\in [1/2,5/2]$ we have
\[
g'(u)\ge x\epsilon+\frac{\min(t,0)}{\pi}\ge \frac{x\epsilon}{2}.
\]
In either case the first derivative bound for oscillatory integrals (see \cite[Lemma 4.3]{Titchmarsh}, for example), applied after integrating by parts to remove the smooth factor $w(u)u^{\sigma-1}$, shows that the integral is $\ll_\sigma 1/(x\epsilon)$. Since we also have the trivial bound $\ll_\sigma 1$ for the integral, and $\min(1,1/(x\epsilon))\ll 1/(1+x\epsilon)$, this gives the third claim.

For the final claim, we see that for $-10\pi x\epsilon\le t\le -\frac{\pi}{2} x\epsilon$ and $u\in[1/2,5/2]$ we have
\[
g''(u)=\frac{|t|}{2\pi u^2}\ge \frac{x\epsilon}{25},
\]
and so by the second derivative bound for oscillatory integrals (see \cite[Lemma 4.5]{Titchmarsh}) the integral is $\ll_\sigma (x\epsilon)^{-1/2}$. Combining this with the trivial bound $\ll_\sigma 1$ as before gives $W_{x,\epsilon}(\sigma+it)\ll_\sigma x^{\sigma}(1+x\epsilon)^{-1/2}$, as desired.
\end{proof}

\begin{rmk}
Lemma \ref{lmm:Saddle} is stated for $\epsilon\ge 0$. If $\epsilon<0$ then, since $w$ is real-valued, we have $W_{x,\epsilon}(\sigma+it)=\overline{W_{x,|\epsilon|}(\sigma-it)}$, and so the same bounds hold with $\epsilon$ replaced by $|\epsilon|$ throughout and with the saddle point region reflected to $\frac{\pi}{2} x|\epsilon|\le t\le 10\pi x|\epsilon|$.
\end{rmk}

\begin{lmm}\label{lmm:Combinatorial}
Let $0\le\eta\le 1/100$, let $k\ge 5$ and let $a_1,\dots,a_k\in[0,1]$ be such that $\sum_{i=1}^k a_i=1$ and $a_i+a_j\ge 2/9-\eta$ for all $i\ne j$. Then at least one of the following holds:
\begin{itemize}
\item There is a subsum of the $a_i$ which lies in $[5/12-\eta,7/12+\eta]$.
\item There is a partition $\mathcal{I}_1\sqcup\mathcal{I}_2\sqcup\mathcal{I}_3$ of $\{1,\dots,k\}$ such that
\[
\frac{5}{12}\ge \sum_{i\in\mathcal{I}_1}a_i\ge \sum_{i\in\mathcal{I}_2}a_i\ge 2\sum_{i\in\mathcal{I}_3}a_i\ge \frac{1}{3}.
\]
\end{itemize}
\end{lmm}
\begin{proof}
Without loss of generality $a_1\ge\dots\ge a_k$. Assume that no subsum lies in $[5/12-\eta,7/12+\eta]$, so that we wish to find a suitable partition of $\{1,\dots,k\}$. First choose $j\ge 0$ maximally such that $a_1 +\dots +a_j< 5/12-\eta$. Then by maximality $a_1+\dots+a_{j+1}> 7/12+\eta$ (since no subsums lie in $[5/12-\eta,7/12+\eta]$), so $a_{j+1}>1/6+2\eta$. Combining this with $a_{j + 1} \le a_{j} \le \dots \le a_{1}$ and $a_{1} + \dots + a_j < 5/12 - \eta$ we obtain $j (1/6 + 2 \eta) < 5/12 - \eta$ and therefore $j \le 2$. 

If $j=0$ then $a_1>7/12+\eta$ and so $a_2+\dots+a_k<5/12-\eta$. But $a_2+a_3\ge 2/9-\eta$ and $a_4+a_5\ge 2/9-\eta$, so, using $\eta\le 1/100<1/36$,
\[
\frac{5}{12} - \eta > a_2+\dots+a_k\ge \frac{4}{9} - 2\eta.
\]
Thus we cannot have $j=0$.

If $j=1$ then $a_1<5/12-\eta$ and $a_1+a_2>7/12+\eta$, so $a_3+\dots+a_k<5/12-\eta$. If $k\ge 6$ then $a_3+a_4\ge 2/9-\eta$ and $a_5+a_6\ge 2/9-\eta$, so $a_3+\dots+a_k\ge 4/9-2\eta$ contradicting $a_3 + \ldots + a_k \le 5/12-\eta$ as before. Hence $k=5$. Since $a_3\ge a_4\ge a_5$ and $a_3+a_4+a_5<5/12-\eta$, we see $a_5<(5/12-\eta)/3\le 5/36$. Thus
\[
a_1+a_3\ge \frac{a_1+a_2+a_3+a_4}{2}=\frac{1-a_5}{2}>\frac{1-5/36}{2}=\frac{31}{72}>\frac{5}{12}-\eta,
\]
so $a_1+a_3>7/12+\eta$. Moreover, $a_1+a_5<(5/12-\eta)+5/36<7/12+\eta$, so $a_1+a_5<5/12-\eta$, so $a_2+a_3+a_4=1-a_1-a_5>7/12+\eta$, so $a_2>7/36+\eta/3$. But then
\[
a_4+a_5=1-(a_1+a_3)-a_2<1-\Bigl(\frac{7}{12}+\eta\Bigr)-\Bigl(\frac{7}{36}+\frac{\eta}{3}\Bigr)=\frac{2}{9}-\frac{4\eta}{3} \le \frac{2}{9}-\eta,
\]
which is impossible since $a_4+a_5\ge 2/9-\eta$. Thus we cannot have $j=1$.

If $j=2$ then we take $\mathcal{I}_1=\{1,2\}$, $\mathcal{I}_2=\{3,4\}$ and $\mathcal{I}_3=\{5,\dots,k\}$, and note that $\sum_{i\in\mathcal{I}_1}a_i=a_1+a_2<5/12-\eta\le 5/12$ and $a_1+a_2\ge a_3+a_4$. We now choose $m\ge 0$ maximally such that $a_3+\dots+a_{2+m}<5/12-\eta$. Since $a_3+a_4\le a_1+a_2<5/12-\eta$ we have $m\ge 2$, while $a_3+\dots+a_k=1-(a_1+a_2)>7/12+\eta$, so $m\le k-3$ and in particular $a_{3+m}$ exists. By maximality and our assumption, $a_3+\dots+a_{3+m}>7/12+\eta$, so $a_{3+m}>1/6+2\eta$; since $a_3\ge\dots\ge a_{3+m}$ this gives $m(1/6+2\eta)<5/12-\eta$, so $m=2$. Thus $a_3+a_4<5/12-\eta$ and $a_3+a_4+a_5>7/12+\eta$, whence $a_5>1/6+2\eta$.

Since $a_5\le (a_3+a_4)/2$ we have $a_3+a_4\ge \frac{2}{3}(a_3+a_4+a_5)>\frac{7}{18}+\frac{2\eta}{3}$, so $a_1+a_2+a_3+a_4\ge 2(a_3+a_4)>\frac{7}{9}+\frac{4\eta}{3}$ and hence $a_5+\dots+a_k<\frac{2}{9}-\frac{4\eta}{3}\le\frac{2}{9}-\eta$. If $k\ge 6$ then $a_5+a_6\ge 2/9-\eta$, a contradiction; hence $k=5$ and $\mathcal{I}_3=\{5\}$. Finally, $a_3+a_4\ge 2a_5$ and $2\sum_{i\in\mathcal{I}_3}a_i=2a_5>1/3+4\eta\ge 1/3$. This gives the required properties of our partition.
\end{proof}

\begin{lmm}[Heath-Brown Identity]\label{lmm:HeathBrown}
Let $x\ge 2$, let $x^{4/9}\le B\le x$, let $\chi$ be a Dirichlet character modulo $q$, and let $s\in\mathbb{C}$. Then
\[
\sum_{n\in [x,2x]}\Lambda(n)\chi(n) n^{-s}=\sum_{j \ll (\log{x})^{10} }c_j F^{\sharp}_j(s,\chi)
\]
where $c_j\ll 1$ and each $F^{\sharp}_j$ is the restriction to $n\in[x,2x]$ of a product $F_j(s,\chi):=\prod_{i=1}^{k_j}S_{i,j}(s,\chi)$ (that is, writing $F_j(s,\chi)=\sum_{n}f_j(n)\chi(n)n^{-s}$, we set $F^{\sharp}_j(s,\chi):=\sum_{n\in[x,2x]}f_j(n)\chi(n)n^{-s}$) for some $k_j\le 10$ and where
\[
S_{i,j}(s,\chi)=\sum_{n\in I_{i,j}}a_{i,j}(n) \chi(n)n^{-s},\qquad I_{i,j}\subseteq [N_{i,j},2^{10}N_{i,j}],
\]
for some lengths $N_{i,j}\ge 1$ and coefficients $a_{i,j}(n)$ which satisfy
\begin{enumerate}[(i)]
\item $\prod_{i=1}^{k_j}N_{i,j}\asymp x$ for all $j$.
\item $N_{i_1,j}N_{i_2,j}\ge 2^{-20}B^{1/2}$ for all $i_1\ne i_2$.
\item $a_{i,j}(n)\ll n^{o(1)}$ for all $n \geq 1$.
\item If $N_{i,j}>B^{1/2}$ we have $I_{i,j}=[N_{i,j},2N_{i,j}]$ and either $a_{i,j}(n)=1$ for all $n$ or $a_{i,j}(n)=\log{n}$ for all $n$.
\end{enumerate}
\end{lmm}
\begin{proof}
The Heath-Brown identity with parameter $K=5$ (see \cite{HeathBrownVaughan}) states that, for all $n\le 2x$,
\[
\Lambda(n)=\sum_{j=1}^{5}(-1)^{j-1}\binom{5}{j}\sum_{\substack{m_1m_2\cdots m_{2j}=n\\ m_{j+1},\dots,m_{2j}\le (2x)^{1/5}}}\mu(m_{j+1})\cdots\mu(m_{2j})\log{m_1}.
\]
Multiplying this by $\chi(n)n^{-s}$, summing over $n\in[x,2x]$, and decomposing each of the (at most $10$) variables of summation into dyadic ranges, we may write the left hand side of the lemma as a linear combination of $\ll (\log{x})^{10}$ restrictions to $n\in[x,2x]$ of products $\prod_{i=1}^{k}S_i(s,\chi)$ with $k\le 10$ (the restriction must be retained, since the condition $m_1\cdots m_{2j}\in[x,2x]$ does not factor through the dyadic decomposition), where each $S_i$ is a Dirichlet polynomial supported on a dyadic interval $[N_i,2N_i]$ with $\prod_{i=1}^kN_i\asymp x$, and the coefficients of each $S_i$ are identically $1$, identically $\log{n}$, or given by the M\"obius function, the latter occurring only for lengths $N_i\le (2x)^{1/5}$. Since $B\ge x^{4/9}$ we have $(2x)^{1/5}\le x^{2/9}\le B^{1/2}$ for $x$ sufficiently large, and so every factor of length greater than $B^{1/2}$ has coefficients identically $1$ or identically $\log{n}$.

It remains to arrange the pairwise condition on the lengths; this concerns only the products, and leaves the restriction to $n\in[x,2x]$ unchanged. Within each product $\prod_{i=1}^{k}S_i(s,\chi)$, as long as there are two factors whose lengths satisfy $N_{i_1}N_{i_2}<2^{-20}B^{1/2}$, we replace them by their product, a single Dirichlet polynomial of length $N_{i_1}N_{i_2}$. Since each product begins with at most $10$ factors and each merge reduces the number of factors by one, the process terminates after at most $9$ steps. At termination every pairwise product $N_{i_1}N_{i_2}$ with $i_1 \neq i_2$ is at least $2^{-20}B^{1/2}$, while any merged polynomial is supported in $[1, 2^{-10}B^{1/2}]$.  The coefficients of a merged polynomial are Dirichlet convolutions of at most $10$ of the original coefficient sequences, hence still of size $\ll n^{o(1)}$. Finally, a polynomial of length greater than $B^{1/2}$ is never produced by a merge, and so is one of the original factors, with dyadic support and coefficients identically $1$ or identically $\log{n}$.
\end{proof}

\begin{lmm}[Mean value theorem]\label{lmm:MVT}
Let $a_n$ be a sequence of complex numbers. 
Let $N,T, q\ge 1$ be given. Then
\[
\sum_{\chi \Mod{q}}\int_{-T}^{T}\Big |\sum_{n \leq N} a_n \chi(n) n^{it} \Big |^2dt\ll (N+qT)\sum_{n\le N}|a_n|^2.
\]
\end{lmm}
\begin{proof}
See \cite[Theorem 6.4]{Montgomery}; see also \cite[Chapter 9]{IwaniecKowalski}.
\end{proof}

\begin{lmm}[Easy Type II estimate]\label{lmm:TypeIIEasy}
Let $x\ge 2$, $T\ge 1$ and $q\ge 1$, and let $F(s,\chi)=\prod_{i=1}^{k}S_{i}(s,\chi)$ be a product of $k\le 10$ Dirichlet polynomials $S_i(s,\chi)=\sum_{n\in I_i}a_{i}(n)\chi(n)n^{-s}$, where $I_i\subseteq[N_i,2^{10}N_i]$, the coefficients satisfy $a_{i}(n)\ll n^{o(1)}$, and the lengths satisfy $\prod_{i=1}^{k}N_i\asymp x$. If there is an $\mathcal{I}\subseteq \{1,\dots,k\}$ such that
\[
x^{5/12}\lessapprox \prod_{i\in \mathcal{I}}N_i\lessapprox x^{7/12}
\]
then we have
\[
\sum_{\chi \Mod{q}}\int_{-T}^{T}|F(it,\chi)|dt\lessapprox x+(qT)^{1/2}x^{19/24}+qTx^{1/2}.
\]
\end{lmm}
\begin{proof}
Let $M_1(s,\chi):=\prod_{i\in\mathcal{I}}S_i(s,\chi)$ and $M_2(s,\chi):=\prod_{i\notin\mathcal{I}}S_i(s,\chi)$ have lengths $M_1\asymp\prod_{i\in \mathcal{I}}N_i$ and $M_2\asymp\prod_{i\notin\mathcal{I}}N_i$ respectively, so $F(s,\chi)=M_1(s,\chi)M_2(s,\chi)$ and $M_1M_2\asymp x$. By assumption on the size of $\prod_{i\in \mathcal{I}}N_i$, $M_1,M_2\gtrapprox x^{5/12}$. Then we have, by Cauchy-Schwarz and Lemma \ref{lmm:MVT},
\begin{align*}
\sum_{\chi \Mod{q}}\int_{-T}^{T}&|F(it,\chi)|dt\\
&\leq \Bigl(\sum_{\chi \Mod{q}}\int_{-T}^{T}|M_1(it,\chi)|^2dt\Bigr)^{1/2}\Bigl(\sum_{\chi \Mod{q}}\int_{-T}^{T}|M_2(it,\chi)|^2dt\Bigr)^{1/2}\\
&\lessapprox (M_1^2+qTM_1)^{1/2}(M_2^2+qTM_2)^{1/2}\\
&\lessapprox x+qTx^{1/2}+(qT)^{1/2}\Bigl(\frac{x}{M_1^{1/2}}+\frac{x}{M_2^{1/2}}\Bigr)\\
&\lessapprox x+(qT)^{1/2}x^{19/24}+qT x^{1/2}.
\end{align*}
Here we used the assumption $M_1,M_2\gtrapprox x^{5/12}$ in the final line.
\end{proof}

\begin{lmm}[Large Values estimate]\label{lmm:LargeValue}
Let $x\ge 2$, $N,T\ge 1$, $q\ge 1$ and $\eta>0$, and let $D(s,\chi)=\sum_{n\in I} a_n \chi(n) n^{-s}$ be a Dirichlet polynomial supported on an interval $I\subseteq[N,4N]$, with coefficients $a_n\ll n^{o(1)}$.  Let $W\subseteq \{\chi \Mod{q}\}\times [-T,T]$ be such that $|D(it,\chi)|\ge N^\sigma$ whenever $(\chi,t)\in W$ and such that if $(\chi_1,t_1),(\chi_2,t_2)\in W$ then either $\chi_1\ne \chi_2$ or $|t_1-t_2|\ge 1$. Then we have
\[
|W|\ll_{\eta} N^{\eta}\Bigl(N^{2-2\sigma}+q^{4/3}T N^{2-4\sigma}+qT N^{12/5-4\sigma}+(qT)^{1/2}N^{3-4\sigma}\Bigr).
\]
In particular, if $N\lessapprox x^{5/12}$ and $qT\lessapprox x^{1/2}$ then
\[
|W|\lessapprox N^{2-2\sigma}+(qT)^{1/2}x^{5/12} N^{2-4\sigma}.
\]
\end{lmm}
\begin{proof}
For $I\subseteq[N,2N]$ this is \cite[Theorem 1.1]{Chen}, extending the method of Guth--Maynard \cite{GuthMaynard} to Dirichlet characters. The extension of \cite{Chen} to support in $[N, 4N]$ instead of $[N, 2N]$ is mundane: split $D=D_1+D_2$ into polynomials supported on dyadic ranges. For each $(\chi,t)\in W$ we have $|D_i(it,\chi)|\ge N^{\sigma}/2$ for some $i\in\{1,2\}$, and applying the dyadic case to $D_1$ and $D_2$ separately, with the threshold $N^{\sigma}/2$, alters the bounds by at most a bounded factor.
\end{proof}

\begin{lmm}[New Type II estimate]\label{lmm:TypeIINew}
Let $x\ge 2$, and let $T,q\ge 1$ satisfy $x^{5/12}<qT\lessapprox x^{1/2}$. Let $M_1,M_2,M_3\ge 1$ satisfy
\[
x^{5/12}\gtrapprox M_1\gtrapprox M_2\gtrapprox M_3^2\gtrapprox x^{1/3},\qquad M_1M_2M_3\asymp x,
\]
and let
\[
M_1(s,\chi):=\sum_{n\sim M_1}\frac{\alpha_n \chi(n)}{n^{s}},\ \ M_2(s,\chi):=\sum_{n\sim M_2}\frac{\beta_n \chi(n)}{n^{s}},\ \ M_3(s,\chi):=\sum_{n\sim M_3}\frac{\gamma_n \chi(n)}{n^{s}}
\]
be Dirichlet polynomials with coefficients satisfying $\alpha_n,\beta_n,\gamma_n\ll n^{o(1)}$. Then
\[
\sum_{\chi \Mod{q}}\int_{-T}^{T}|M_1(it,\chi)M_2(it,\chi)M_3(it,\chi)|dt\lessapprox  x+(qT)^{7/8}x^{29/48}.
\]
\end{lmm}
\begin{proof}
By splitting the summation over $\chi$ and integral over $t$ into level sets, it suffices to show that for any choice of $\sigma_1,\sigma_2,\sigma_3\le 1$ and any 1-separated set $W\subset \{\chi\Mod{q}\}\times[-T,T]$ such that if $(\chi,t)\in W$ then $|M_1(it,\chi)|\ge M_1^{\sigma_1}$, $|M_2(it,\chi)|\ge M_2^{\sigma_2}$ and $|M_3(it,\chi)|\ge M_3^{\sigma_3}$ we have
\begin{equation}
|W|\lessapprox M_1^{-\sigma_1}M_2^{-\sigma_2}M_3^{-\sigma_3}\Bigl( x+(qT)^{7/8}x^{29/48}\Bigr).
\label{eq:Target}
\end{equation}
Let $\sigma$ be defined such that $x^\sigma=M_1^{\sigma_1}M_2^{\sigma_2}M_3^{\sigma_3}$. By Lemma \ref{lmm:MVT} applied to $M_1(s,\chi)$, $M_2(s,\chi)$ and $M_3(s,\chi)^2$ (noting that these all have length $\lessapprox qT$) we have
\begin{align}
|W|&\lessapprox qT \min\Bigl(M_1^{1-2\sigma_1},M_2^{1-2\sigma_2},M_3^{2-4\sigma_3}\Bigr)\nonumber\\
&\lessapprox qT \Bigl(M_1^{1-2\sigma_1}\Bigr)^{2/5}\Bigl(M_2^{1-2\sigma_2}\Bigr)^{2/5}\Bigl(M_3^{2-4\sigma_3}\Bigr)^{1/5}\nonumber\\
&\lessapprox qT x^{(2-4\sigma)/5}.\label{eq:MVTBound}
\end{align}
By Lemma \ref{lmm:LargeValue} we have
\[
|W|\lessapprox \max\Bigl(N^{2-2\beta}, (qT)^{1/2}x^{5/12}N^{2-4\beta}\Bigr)
\] 
for any choice of $(N,\beta)\in \{(M_1,\sigma_1),\allowbreak (M_2,\sigma_2),\allowbreak (M_3^2,\sigma_3)\}$. We now do some rather tedious casework depending on which of the two terms in the maximum above is the larger, for each of the three choices of $(N,\beta)$, giving eight cases in total. In each case $|W|$ is bounded by the minimum of the three bounds that hold, and we bound this minimum by a weighted geometric mean $A_1^{\theta_1}A_2^{\theta_2}A_3^{\theta_3}$ (with $\theta_1+\theta_2+\theta_3=1$), the weights being chosen so that the exponents of $\sigma_1,\sigma_2,\sigma_3$ align to give a bound in terms of $x$ and $\sigma$ alone. If the exponent of $x$ in the resulting expression is of the form $c_0-c\sigma$ with $c>1$, we then combine this bound with \eqref{eq:MVTBound} so as to arrive at a bound in which $c=1$. In each case we will end up with a bound of the form $|W|\lessapprox x^{1-\sigma}$ or $|W|\lessapprox (qT)^{7/8}x^{29/48-\sigma}$; since $x^{\sigma}=M_1^{\sigma_1}M_2^{\sigma_2}M_3^{\sigma_3}$, either bound implies \eqref{eq:Target}, and hence the lemma will follow.

\textbf{Case 1: }$|W|\lessapprox \min\Bigl(M_1^{2-2\sigma_1},M_2^{2-2\sigma_2},M_3^{4-4\sigma_3}\Bigr)$.
In this case we have
\[
|W|\lessapprox \Bigl(M_1^{2-2\sigma_1}\Bigr)^{2/5}\Bigl(M_2^{2-2\sigma_2}\Bigr)^{2/5}\Bigl(M_3^{4-4\sigma_3}\Bigr)^{1/5}\lessapprox x^{(4-4\sigma)/5}\lessapprox x^{1-\sigma}.
\]
\textbf{Case 2: }$|W|\lessapprox \min\Bigl(M_1^{2-2\sigma_1},M_2^{2-2\sigma_2}, (qT)^{1/2}x^{5/12}M_3^{4-8\sigma_3}\Bigr)$.
In this case we have
\begin{align*}
|W|&\lessapprox \Bigl(M_1^{2-2\sigma_1}\Bigr)^{4/9}\Bigl(M_2^{2-2\sigma_2}\Bigr)^{4/9}\Bigl((qT)^{1/2}x^{5/12}M_3^{4-8\sigma_3}\Bigr)^{1/9}\\
&= \frac{(qT)^{1/18}x^{5/108}x^{(8-8\sigma)/9}}{M_3^{4/9}}\lessapprox \Bigl(\frac{(qT)^6x^5}{M_3^{48}}\Bigr)^{1/108}x^{1-\sigma}.
\end{align*}
Since $qT\lessapprox x^{1/2}$ and $M_3\gtrapprox x^{1/6}$ this gives $|W|\lessapprox x^{1-\sigma}$.

\textbf{Case 3: }$|W|\lessapprox \min\Bigl(M_1^{2-2\sigma_1},(qT)^{1/2}x^{5/12}M_2^{2-4\sigma_2}, M_3^{4-4\sigma_3}\Bigr)$.
In this case we have
\[
|W|\lessapprox \Bigl(M_1^{2-2\sigma_1}\Bigr)^{1/2}\Bigl((qT)^{1/2}x^{5/12}M_2^{2-4\sigma_2}\Bigr)^{1/4}\Bigl( M_3^{4-4\sigma_3}\Bigr)^{1/4}=\frac{x^{5/48}(qT)^{1/8}}{M_2^{1/2}}x^{1-\sigma}.
\]
Since $M_2\gtrapprox x^{1/3}$ and $qT\lessapprox x^{1/2}$ this gives $|W|\lessapprox x^{1-\sigma}$.

\textbf{Case 4: }$|W|\lessapprox \min\Bigl((qT)^{1/2}x^{5/12}M_1^{2-4\sigma_1},M_2^{2-2\sigma_2}, M_3^{4-4\sigma_3}\Bigr)$.
This is analogous to Case 3, giving (recalling $M_2\lessapprox M_1$)
\[
|W|\lessapprox \frac{(qT)^{1/8}x^{5/48}}{M_1^{1/2}}x^{1-\sigma}\lessapprox \frac{(qT)^{1/8}x^{5/48}}{M_2^{1/2}}x^{1-\sigma}\lessapprox x^{1-\sigma}.
\]
\textbf{Case 5: }$|W|\lessapprox \min\Bigl((qT)^{1/2}x^{5/12}M_1^{2-4\sigma_1},(qT)^{1/2}x^{5/12}M_2^{2-4\sigma_2}, M_3^{4-4\sigma_3}\Bigr)$.
In this case we have
\begin{align*}
|W|&\lessapprox\Bigl((qT)^{1/2}x^{5/12}M_1^{2-4\sigma_1}\Bigr)^{1/3}\Bigl((qT)^{1/2}x^{5/12}M_2^{2-4\sigma_2}\Bigr)^{1/3}\Bigl( M_3^{4-4\sigma_3}\Bigr)^{1/3}\\
&\lessapprox M_3^{2/3}(qT)^{1/3}x^{5/18+(2-4\sigma)/3}.
\end{align*}
We then combine this with the bound \eqref{eq:MVTBound} and use $M_3\lessapprox M_1^{1/2}\lessapprox x^{5/24}\le (qT)^{1/2}$, giving
\begin{align*}
|W|&\lessapprox \Bigl(qT x^{(2-4\sigma)/5}\Bigr)^{5/8}\Bigl(M_3^{2/3}(qT)^{1/3}x^{5/18+(2-4\sigma)/3}\Bigr)^{3/8}\\
&=M_3^{1/4}(qT)^{3/4}x^{29/48-\sigma}\lessapprox (qT)^{7/8}x^{29/48-\sigma}.
\end{align*}
\textbf{Case 6: }$|W|\lessapprox \min\Bigl(M_1^{2-2\sigma_1},(qT)^{1/2}x^{5/12}M_2^{2-4\sigma_2}, (qT)^{1/2}x^{5/12}M_3^{4-8\sigma_3}\Bigr)$.
In this case we have
\begin{align*}
|W|&\lessapprox\Bigl(M_1^{2-2\sigma_1}\Bigr)^{4/7}\Bigl((qT)^{1/2}x^{5/12}M_2^{2-4\sigma_2}\Bigr)^{2/7}\Bigl((qT)^{1/2}x^{5/12} M_3^{4-8\sigma_3}\Bigr)^{1/7}\\
&\lessapprox(qT)^{3/14}x^{5/28+(4-8\sigma)/7}M_1^{4/7}.
\end{align*}
We then combine this with the bound \eqref{eq:MVTBound} and use $M_1\lessapprox x^{5/12}< qT$, giving
\begin{align*}
|W|&\lessapprox \Bigl(qT x^{(2-4\sigma)/5}\Bigr)^{5/12}\Bigl((qT)^{3/14}x^{5/28+(4-8\sigma)/7}M_1^{4/7}\Bigr)^{7/12}\\
&=M_1^{1/3}(qT)^{13/24}x^{29/48-\sigma}\lessapprox (qT)^{7/8}x^{29/48-\sigma}.
\end{align*}
\textbf{Case 7: }$|W|\lessapprox \min\Bigl((qT)^{1/2}x^{5/12}M_1^{2-4\sigma_1},M_2^{2-2\sigma_2}, (qT)^{1/2}x^{5/12}M_3^{4-8\sigma_3}\Bigr)$.
This is analogous to Case 6, with the roles of $M_1$ and $M_2$ interchanged (and using $M_2\lessapprox M_1$), and we have
\[
|W|\lessapprox M_2^{1/3}(qT)^{13/24}x^{29/48-\sigma}\lessapprox M_1^{1/3}(qT)^{13/24}x^{29/48-\sigma}\lessapprox (qT)^{7/8}x^{29/48-\sigma}.
\]
\textbf{Case 8: }$|W|\lessapprox (qT)^{1/2}x^{5/12}\min\Bigl(M_1^{2-4\sigma_1},M_2^{2-4\sigma_2}, M_3^{4-8\sigma_3}\Bigr)$.
In this case we have
\begin{align*}
|W|&\lessapprox (qT)^{1/2}x^{5/12}\Bigl(M_1^{2-4\sigma_1}\Bigr)^{2/5}\Bigl(M_2^{2-4\sigma_2}\Bigr)^{2/5}\Bigl(M_3^{4-8\sigma_3}\Bigr)^{1/5}\\
&\lessapprox (qT)^{1/2}x^{5/12}x^{(4-8\sigma)/5}.
\end{align*}
We then combine this with the bound \eqref{eq:MVTBound}, giving
\begin{align*}
|W|&\lessapprox \Bigl(qT x^{(2-4\sigma)/5}\Bigr)^{3/4}\Bigl((qT)^{1/2}x^{5/12}x^{(4-8\sigma)/5}\Bigr)^{1/4}\\
&=(qT)^{7/8}x^{29/48-\sigma}.
\end{align*}
\end{proof}

\begin{lmm}[Long factors]\label{lmm:Long}
Let $x\ge 2$, $x^{4/9}\le B\le x$, $1\le T\ll x$ and $1\le q\le x$ with $B\lessapprox qT$. Let $F(s,\chi)=S_1(s,\chi)\cdots S_k(s,\chi)$, with $k\le 4$, be one of the products $F_j(s,\chi)$ appearing in Lemma \ref{lmm:HeathBrown}. Then we have
\[
\sup_{|t_0|<x}\sum_{\chi \Mod{q}}\int_{-T}^{T}|F(i(t+t_0),\chi)|dt\lessapprox (qT) x^{1/2}+x.
\]
\end{lmm}
\begin{proof}
Recall from Lemma \ref{lmm:HeathBrown} that if $S_{j}$ has length $N_j>B^{1/2}$ then either $S_j(s,\chi)=\sum_{n\sim N_j}\chi(n)n^{-s}$ or $S_j(s,\chi)=\sum_{n\sim N_j}(\log{n})\chi(n)n^{-s}$. Let $\tilde{L}(s,\chi)=L(s,\chi)$ in the first case, and $\tilde{L}(s,\chi)=-L'(s,\chi)$ in the second case (so that $S_j$ is a sum of terms of $\tilde{L}$ with $n\sim N_j$). By the truncated Perron formula with $U:=(qx)^3$
\begin{align*}
S_j(s,\chi)&=\frac{1}{2\pi i}\int_{2-iU}^{2+iU}\frac{(2N_j)^u-N_j^u}{u}\tilde{L}(s+u,\chi)du + O(x^{o(1)}).
\end{align*}
Moving the line of integration to $\Re(u)=1/2$ and picking up a possible pole from $u=1-s$ when $\chi=\chi_0$, we find that whenever $N_j>B^{1/2}$ and $|t|<U$
\[
S_j(it,\chi)=P_j(it,\chi)+R_j(it,\chi),
\]
where
\begin{align}
P_j(it,\chi)&:=\Res_{u=1-it}\Bigl(\frac{N_j^u(2^u-1)\tilde{L}(it+u,\chi)}{u}\Bigr),\nonumber\\
R_j(it,\chi)&:=\frac{N_j^{1/2}}{2\pi}\int_{-U}^{U}\frac{2^{1/2}(2N_j)^{iu}-N_j^{iu}}{1/2+iu}\tilde{L}(1/2+i(t+u),\chi)du+O(x^{o(1)}).\label{eq:RjDef}
\end{align}
Here we used the convexity bound $\tilde{L}(1/2+it)\lessapprox (qU)^{1/4}$ to bound the horizontal contours by $\lessapprox N_j^2(qU)^{1/4}/U\lessapprox 1$.
We see that $P_j(it,\chi)\lessapprox \mathbf{1}_{\chi=\chi_0}N_j/(1+|t|)$. Since trivially $S_j\lessapprox N_j$, we also have $R_j\lessapprox N_j$. Let $J\subseteq\{1,\dots,k\}$ be the set of $j$ with $N_j>B^{1/2}$. By expanding the product and bounding pointwise any term involving any $P_j$, we have
\begin{align*}
\prod_{j=1}^k |S_j(i t,\chi)| &=\prod_{j\in J}\Bigl(|P_j(i t,\chi)|+|R_j(i t,\chi)|\Bigr)\prod_{j\notin J}|S_j(i t,\chi)|\\
&\lessapprox \frac{\mathbf{1}_{\chi=\chi_0}x}{1+|t|}+\prod_{j\in J}|R_j(i t,\chi)|\prod_{j\notin J}|S_j(i t,\chi)|.
\end{align*}
Here we used the bounds $|R_j|\lessapprox N_j$, $|S_j|\lessapprox N_j$, $|P_j|\lessapprox N_j/(1+|t|)$ and $\prod_j N_j\asymp x$ in bounding the first term above. Substituting this into the expression of the lemma, we find that for $|t_0|<x$
\begin{align}
\sum_{\chi \Mod{q}}&\int_{-T}^{T}|F(i(t+t_0),\chi)|dt\lessapprox x\int_{-T}^T\frac{dt}{1+|t_0+t|}\nonumber\\
&\qquad+\sum_{\chi \Mod{q}}\int_{-T}^{T}\prod_{j\in J}|R_j(i (t+t_0),\chi)|\prod_{j\notin J}|S_j(i(t+t_0),\chi)|dt.\label{eq:ExpandedProduct}
\end{align}
The first term on the right hand side is $\lessapprox x$. By H\"older's inequality for any functions $D_j(s,\chi)$ we have (recalling $k\le 4$)
\[
\sum_{\chi \Mod{q}}\int_{-T}^{T}\prod_{j=1}^k |D_j(it,\chi)|dt\ll (qT)^{(4-k)/4}\prod_{j=1}^k\Big(\sum_{\chi \Mod{q}}\int_{-T}^{T} |D_j(it,\chi)|^4 dt\Bigr)^{1/4}.
\]
Applying this to the second term of \eqref{eq:ExpandedProduct}, we see this term is $\lessapprox qT x^{1/2}$ provided
\begin{align*}
\sum_{\chi \Mod{q}}\int_{-T}^{T}|S_j(i (t+t_0),\chi)|^4dt&\lessapprox qT N_j^2\qquad (j\notin J),\\
\sum_{\chi \Mod{q}}\int_{-T}^{T}|R_j(i (t+t_0),\chi)|^4dt&\lessapprox qT N_j^2\qquad (j\in J).
\end{align*}
The first of these claims follows from Lemma \ref{lmm:MVT} applied to $S_j(s+it_0,\chi)^2$ since $N_j^2\le B\lessapprox qT$ when $j\notin J$. For the second claim, by substituting \eqref{eq:RjDef}, we have
\[
\sum_{\chi \Mod{q}}\int_{-T}^{T}|R_j(i (t+t_0),\chi)|^4dt\lessapprox  qT + N_j^2\sum_{\chi \Mod{q}}\int_{t_0-U-T}^{t_0+U+T}|\tilde{L}(1/2+iv,\chi)|^4 h(v)dv
\]
where $h(v)=\int_{-T}^Tdt/(1+|v+t|)\lessapprox \min(1,T/|v|)$. The $4^{th}$ moment bound (see \cite[Chapter 10]{Montgomery}, which also applies to $L'(s,\chi)$) shows that for $V\ge 1$ 
\[
\sum_{\chi\Mod{q}}\int_{-V}^{V}|\tilde{L}(1/2+iv,\chi)|^4dv\lessapprox qV.
\]
After splitting our integral into dyadic regions, this then gives the bound 
\[
\sum_{\chi \Mod{q}}\int_{-T}^{T}|R_j(i (t+t_0),\chi)|^4dt\lessapprox qT + N_j^2 \sup_{V\ll (qx)^3}\min(1,T/V)qV\lessapprox N_j^2qT,
\]
as desired. This gives the result.
\end{proof}

\begin{lmm}[Short factors]\label{lmm:Short}
Let $x\ge 2$, $T\ge 1$,  $q\ge 1$ $x^{4/9}\le B\le x$ satisfy $B\lessapprox qT$. Let $F(s,\chi)=S_1(s,\chi)\cdots S_k(s,\chi)$, with $k\ge 5$, be one of the products $F_j(s,\chi)$ appearing in Lemma \ref{lmm:HeathBrown}. Let $qT\lessapprox x^{1/2}$. Then we have
\[
\sup_{t_0}\sum_{\chi \Mod{q}}\int_{-T}^{T}|F(it+it_0,\chi)|dt\lessapprox x+(qT)^{1/2} x^{19/24}.
\]
\end{lmm}
\begin{proof}
We absorb the factors $n^{-it_0}$ into the coefficients $a_j(n)$ of $S_j(s,\chi)$ (noting this doesn't affect the bound $a_j(n)\ll n^{o(1)}$), so we may assume that $t_0=0$ but can no longer assume any smoothness properties of the $a_j(n)$. Let $N:=\prod_{j=1}^kN_j\asymp x$ and $N_j=N^{\alpha_j}$ for some reals $\alpha_1,\dots,\alpha_k\in [0,1]$ with $\sum_{j=1}^k\alpha_j=1$.

By Lemma \ref{lmm:HeathBrown} we have $N_iN_j\gg B^{1/2}$ for all $i\ne j$, and so, recalling that $B\ge x^{4/9}$, there is an $\eta=o(1)$ such that
\[
\alpha_i+\alpha_j\ge \frac{2}{9}-\eta\qquad\text{for all }i\ne j.
\]
By Lemma \ref{lmm:Combinatorial}, applied with this $\eta$ (which satisfies $\eta\le 1/100$ for $x$ sufficiently large), either there is a subset $\mathcal{J}\subseteq\{1,\dots,k\}$ such that $\sum_{j\in\mathcal{J}}\alpha_j\in[5/12-\eta,7/12+\eta]$, or there is a partition $\{1,\dots,k\}=\mathcal{I}_1\sqcup\mathcal{I}_2\sqcup\mathcal{I}_3$ with $5/12\ge \sum_{i\in\mathcal{I}_1}\alpha_i\ge \sum_{i\in\mathcal{I}_2}\alpha_i\ge 2\sum_{i\in\mathcal{I}_3}\alpha_i\ge 1/3$.

In the first case, since $N\asymp x$ and $x^{\eta}=x^{o(1)}$, we have that 
\[
x^{5/12}\lessapprox \prod_{j\in\mathcal{J}}N_j\lessapprox x^{7/12}.
\]
Therefore Lemma \ref{lmm:TypeIIEasy} gives the bound
\[
\lessapprox x+(qT)^{1/2}x^{19/24}+qTx^{1/2}\lessapprox x+(qT)^{1/2}x^{19/24},
\]
where we used $qT\lessapprox x^{1/2}$ in the final step.

In the second case, since $N\asymp x$, we have that
\[
x^{5/12}\gg \prod_{i\in\mathcal{I}_1}N_i \ge \prod_{i\in\mathcal{I}_2}N_i \ge\prod_{i\in\mathcal{I}_3}N_i^2\gg x^{1/3}.
\]
Each product $\prod_{i\in \mathcal{I}_r}S_i(s,\chi)$ is supported on $[P_r,2^{10k}P_r]$, where $P_r:=\prod_{i\in\mathcal{I}_r}N_i$, and so may be written as a sum of $O(1)$ Dirichlet polynomials supported on dyadic ranges, with coefficients $\ll n^{o(1)}$. By the triangle inequality, Lemma \ref{lmm:TypeIINew} (whose hypothesis $qT>x^{5/12}$ holds, since $qT\gtrapprox B\ge x^{4/9}$ and $4/9>5/12$), applied with $M_1(s,\chi)$, $M_2(s,\chi)$ and $M_3(s,\chi)$ ranging over the dyadic pieces of the three products, gives the bound $\lessapprox x+(qT)^{7/8}x^{29/48}$. This is again $\lessapprox x+(qT)^{1/2}x^{19/24}$, since $(qT)^{7/8}x^{29/48}\lessapprox (qT)^{1/2}x^{19/24}$ when $qT\lessapprox x^{1/2}$ (the two terms crossing at $qT=x^{1/2}$).
\end{proof}

\begin{lmm}[Removal of the cutoff]\label{lmm:Removal}
Let $x\ge 2$, $1\le T\ll x$ and $q\ge 1$. Let $F(s,\chi)=\prod_{i=1}^{k}S_i(s,\chi)$ be one of the products $F_j(s,\chi)$ appearing in Lemma \ref{lmm:HeathBrown}, whose parameter $B$ satisfies $B\lessapprox qT$, and let $F^{\sharp}$ denote its restriction to $n\in[x,2x]$, as there. If $k\le 4$ then
\[
\sum_{\chi \Mod{q}}\int_{-T}^{T}|F^{\sharp}(it,\chi)|dt\lessapprox (qT)x^{1/2}+x,
\]
while if $k\ge 5$ and $qT\lessapprox x^{1/2}$ then
\[
\sum_{\chi \Mod{q}}\int_{-T}^{T}|F^{\sharp}(it,\chi)|dt\lessapprox x+(qT)^{1/2}x^{19/24}.
\]
\end{lmm}
\begin{proof}
Since $F(s,\chi)$ is entire and $|F(s,\chi)|\ll x^{\Re(\sigma)}$, Perron's formula on the line $\Re(w)=0$, truncated at height $x$, gives
\[
F^{\sharp}(it,\chi)=\frac{1}{2\pi}\int_{-x}^{x}F(i(t+u),\chi)\kappa(u)du+O(x^{o(1)}),
\]
where $\kappa(u):=\bigl((2x)^{iu}-x^{iu}\bigr)/(iu)$ satisfies $|\kappa(u)|\ll \min(1,|u|^{-1})$. Since $\int_{-x}^x|\kappa(u)|du\lessapprox 1$, this gives
\begin{align*}
\sum_{\chi \Mod{q}}\int_{-T}^{T}|F^{\sharp}(it,\chi)|dt&\lessapprox \sup_{|u|\le x}\sum_{\chi \Mod{q}}\int_{-T}^{T}|F(i(t+u),\chi)|dt+qT.
\end{align*}
If $k\le 4$ then Lemma \ref{lmm:Long} shows that this is $\ll (qT)x^{1/2}+x$, whereas if $k\ge 5$ and $qT\lessapprox x^{1/2}$ then Lemma \ref{lmm:Short} shows that this is $\lessapprox x+(qT)^{1/2}x^{19/24}$. This gives the result.
\end{proof}

\begin{prpstn}\label{prpstn:Dyadic}
Let $\alpha=a/q+\epsilon$ with $(a,q)=1$ and $B:=\max(q,qx|\epsilon|)\in [x^{4/9},x^{1/2}]$. Then we have
\begin{align*}
\Bigl|\sum_{n\sim x} \Lambda(n)e(\alpha n)\Bigr|&\lessapprox \frac{x}{B^{1/2}}+x^{19/24}.
\end{align*}
\end{prpstn}
\begin{proof}
We may assume that $\epsilon\ge 0$: since $\Lambda$ is real-valued, the sums $\sum_{n\sim x}\Lambda(n)e(\alpha n)$ and $\sum_{n\sim x}\Lambda(n)e(-\alpha n)$ are complex conjugates of one another, and $-\alpha=(q-a)/q-\epsilon$ is an approximation of the same shape, with $(q-a,q)=1$ and the same value of $B$.

Next, we note that the terms with $(n,q)\ne 1$ have $n=p^j$ for some prime $p\mid q$ and integer $j\ge 1$, and so contribute at most
\[
\sum_{\substack{p\mid q,\ j\ge 1\\ p^j\sim x}}\log{p}\lessapprox 1,
\]
which is negligible. Thus we may restrict to $(n,q)=1$. Let $w$ be a smooth function supported on $[1/2,5/2]$ which is equal to $1$ on $[1,2]$ and satisfies $\|w^{(j)}\|_\infty\ll_j 1$. Then we can insert a factor $w(n/x)$ into the sum without changing anything. Thus
\[
\sum_{n\sim x} \Lambda(n)e(\alpha n)=\sum_{\substack{n\sim x\\ (n,q)=1}} w\Bigl(\frac{n}{x}\Bigr)\Lambda(n)e(\alpha n)+O(x^{o(1)}).
\]
Since $\alpha=a/q+\epsilon$, splitting into residue classes $\Mod{q}$ via Dirichlet characters gives for $(n,q)=1$
\begin{align*}
e(\alpha n)&=\frac{1}{\phi(q)}\sum_{\chi\Mod{q}}\sum_{b\Mod{q}}e\Bigl(\frac{b a}{q}\Bigr)\overline{\chi}(b)\chi(n) e(n\epsilon)\\
&=\sum_{\chi\Mod{q}}\frac{\tau(\overline{\chi})\chi(a)}{\phi(q)}\chi(n)e(n\epsilon).
\end{align*}
Here $|\tau(\overline{\chi})|\le q^{1/2}$ is the Gauss sum. Thus
\begin{align*}
\Bigl|\sum_{\substack{n\sim x\\ (n,q)=1}} w\Bigl(\frac{n}{x}\Bigr)\Lambda(n)e(\alpha n)\Bigr|&\le \frac{q^{1/2}}{\phi(q)}\sum_{\chi\Mod{q}}\Bigl|\sum_{n\sim x}\chi(n)\Lambda(n)e(n\epsilon)w\Bigl(\frac{n}{x}\Bigr)\Bigr|.
\end{align*}
By Mellin inversion, we have that
\[
e(n\epsilon)w\Bigl(\frac{n}{x}\Bigr)=\frac{1}{2\pi i}\int_{-i\infty}^{+i\infty} n^{-s}W_{x,\epsilon}(s)ds.
\]
Thus we have 
\begin{align*}
\sum_{n\sim x}\chi(n)\Lambda(n)e(n\epsilon)w\Bigl(\frac{n}{x}\Bigr)&=\frac{1}{2\pi}\int_{-\infty}^{\infty}\Bigl(\sum_{n\sim x}\frac{\chi(n)\Lambda(n)}{n^{it}}\Bigr)W_{x,\epsilon}(it)dt.
\end{align*}
By Lemma \ref{lmm:Saddle}, there is an $\eta := \eta(x)>0$ going to zero as $x \rightarrow \infty$ such that the contribution from $|t|>x^{1+\eta}\epsilon+x^\eta$ and $|t|<x^{1-\eta}\epsilon-1$ is $O(x^{-100})$, which is negligible. For the remaining $t$, we have $W_{x,\epsilon}(it)\lessapprox (1+x\epsilon)^{-1/2}$. Set $T:=x^{\eta}(x\epsilon+1) = x^{o(1)} (x \epsilon + 1)$, so that the remaining range is $\{t:\,x^{1-\eta}\epsilon-1\le |t|\le T\}\subseteq[-T,T]$, and note that $1\le T\ll x$ and $(1+x\epsilon)^{-1/2}\lessapprox T^{-1/2}$. Thus, summing over $\chi$, we have
\begin{align*}
\sum_{\chi\Mod{q}}\Bigl|\sum_{n\sim x}\chi(n)\Lambda(n)&e(n\epsilon)w\Bigl(\frac{n}{x}\Bigr)\Bigr|\\
&\lessapprox\frac{1}{T^{1/2}}\sum_{\chi\Mod{q}}\int_{-T}^{T}\Bigl|\sum_{n\sim x}\frac{\chi(n)\Lambda(n)}{n^{it}}\Bigr|dt+O(x^{-100}).
\end{align*}
By the Heath-Brown identity with the parameter $B:=\max(q,qx|\epsilon|)$ (so that $x^{4/9}\le B\le x^{1/2}$ by assumption), Lemma \ref{lmm:HeathBrown} shows that
\[
\sum_{\chi\Mod{q}}\int_{-T}^{T}\Bigl|\sum_{n\sim x}\frac{\chi(n)\Lambda(n)}{n^{it}}\Bigr|dt\lessapprox \sup_j \sum_{\chi\Mod{q}}\int_{-T}^{T}|F^{\sharp}_j(it,\chi)|dt
\]
for one of the functions $F^{\sharp}_{j}$ described in Lemma \ref{lmm:HeathBrown}. Lemma \ref{lmm:Removal} then shows that
\[
\sum_{\chi\Mod{q}}\int_{-T}^{T}|F^{\sharp}_j(it,\chi)|dt\lessapprox (qT)x^{1/2}+x+(qT)^{1/2}x^{19/24}\lessapprox x+(qT)^{1/2}x^{19/24}
\]
provided $B\lessapprox qT\lessapprox x^{1/2}$; both hold since $B \asymp q(x\epsilon+1)$ and $qT = x^{o(1)}\,q(1 + x \epsilon)\asymp x^{o(1)}B$, since $\eta = o(1)$. Putting this together, we find
\begin{align*}
\Bigl|\sum_{\substack{n\sim x\\ (n,q)=1}} w\Bigl(\frac{n}{x}\Bigr)\Lambda(n)e(\alpha n)\Bigr|&\lessapprox  \sup_{j}\frac{1}{(qT)^{1/2}}\sum_{\chi\Mod{q}}\int_{-T}^{T}|F^{\sharp}_j(it,\chi)|dt+x^{-99}\\
&\lessapprox \frac{x}{(qT)^{1/2}}+x^{19/24}\\
&\lessapprox \frac{x}{q^{1/2}(1+x|\epsilon|)^{1/2}}+x^{19/24},
\end{align*}
as required.
\end{proof}

\section{Proof of Theorem \ref{thrm:MainThm}}\label{sec:Deduction}
In this section we deduce Theorem \ref{thrm:MainThm} from Proposition \ref{prpstn:Dyadic}. We first record an estimate of Kumchev, which will handle the dyadic ranges where the Diophantine quantity $B$ is small.

\begin{lmm}[Kumchev's estimate]\label{lmm:Kumchev}
Let $x\ge 2$ and let $\alpha=a/q+\epsilon$ with $(a,q)=1$, $q\le x^{1/2}$ and $|\epsilon|\le 1/(qx^{1/2})$, and set $B:=\max(q,qx|\epsilon|)$. Then we have
\[
\Bigl|\sum_{n\sim x}\Lambda(n)e(\alpha n)\Bigr|\lessapprox \frac{x}{B^{1/2}}+B^{1/2}x^{11/20}.
\]
In particular, if $B\le x^{9/20}$ then the right hand side is $\lessapprox x/B^{1/2}$.
\end{lmm}
\begin{proof}
This follows from the case $k=1$ of \cite[Theorem 2]{Kumchev}. 
The final claim follows on noting that $B^{1/2}x^{11/20}\le x/B^{1/2}$ precisely when $B\le x^{9/20}$.
\end{proof}

\begin{proof}[Proof of Theorem \ref{thrm:MainThm}]
We may assume throughout that $N$ is sufficiently large. First note that the hypotheses $q\le N^{1/2}$ and $|\epsilon|\le 1/(qN^{1/2})$ imply that
\begin{equation}
B=\max(q,qN|\epsilon|)\le N^{1/2},\qquad\text{and so}\qquad \frac{N}{B^{1/2}}\ge N^{3/4}.
\label{eq:BUpper}
\end{equation}
In particular, any quantity which is $O(N^{3/4+o(1)})$ is acceptable for Theorem \ref{thrm:MainThm}.

Let $J:=\lceil \frac{\log N}{4\log{2} }\rceil$, so that $N_0:=N/2^{J}$ satisfies $N_0\in[N^{3/4}/2,N^{3/4}]$. Then we have
\[
\sum_{n<N}\Lambda(n)e(n\alpha)=\sum_{j=1}^{J}\,\sum_{n\sim N/2^{j}}\Lambda(n)e(n\alpha)+O\Bigl(\sum_{n\le N_0}\Lambda(n)\Bigr),
\]
where the first error term accounts for the range $n\le N_0$ and the second for the boundary term $n=N$. By Chebyshev's estimate we have $\sum_{n\le N_0}\Lambda(n)\ll N_0\ll N^{3/4}$, which is acceptable by \eqref{eq:BUpper}. Since there are $O(\log{N})$ dyadic ranges, it therefore suffices to show that
\begin{equation}
\Bigl|\sum_{n\sim x}\Lambda(n)e(n\alpha)\Bigr|\lessapprox \frac{N}{B^{1/2}}+N^{19/24}
\label{eq:BlockTarget}
\end{equation}
for each $x\in[N^{3/4}/2,N/2]$.

Fix such an $x$. By Dirichlet's approximation theorem there exist coprime integers $a',q'$ with
\[
1\le q'\le x^{1/2},\qquad \Bigl|\alpha-\frac{a'}{q'}\Bigr|\le \frac{1}{q'x^{1/2}}.
\]
Write $\epsilon':=\alpha-a'/q'$ and $B':=\max(q',q'x|\epsilon'|)$, and note that $B'\le x^{1/2}$. We claim that
\begin{equation}
\Bigl|\sum_{n\sim x}\Lambda(n)e(n\alpha)\Bigr|\lessapprox \frac{x}{(B')^{1/2}}+x^{19/24}.
\label{eq:BlockBound}
\end{equation}
Indeed, if $B'\ge x^{4/9}$ then \eqref{eq:BlockBound} is precisely Proposition \ref{prpstn:Dyadic}, applied with $a',q',\epsilon'$ in place of $a,q,\epsilon$. If instead $B'<x^{4/9}$ then, since $4/9<9/20$, Lemma \ref{lmm:Kumchev} gives the stronger bound $\lessapprox x/(B')^{1/2}$. This proves \eqref{eq:BlockBound}.

Since $x\le N$ the term $x^{19/24}$ is acceptable for \eqref{eq:BlockTarget}, and so it suffices to show that
\begin{equation}
\frac{x}{(B')^{1/2}}\ll \frac{N}{B^{1/2}}+N^{3/4}.
\label{eq:Consistency}
\end{equation}

\textbf{Case 1: $a'/q'=a/q$.} In this case $q'=q$ and $\epsilon'=\epsilon$. Since $\max(q,qx|\epsilon|)\ge \frac{x}{N}q$ and $qx|\epsilon|=\frac{x}{N}\cdot qN|\epsilon|$, we have
\[
B'=\max(q,qx|\epsilon|)\ge \frac{x}{N}\max(q,qN|\epsilon|)=\frac{x}{N}B,
\]
and hence
\[
\frac{x}{(B')^{1/2}}\le x\Bigl(\frac{N}{xB}\Bigr)^{1/2}=\frac{(xN)^{1/2}}{B^{1/2}}\le \frac{N}{B^{1/2}}.
\]
\textbf{Case 2: $a'/q'\ne a/q$.} In this case $aq'-a'q$ is a non-zero integer, so
\begin{equation}
\frac{1}{qq'}\le \Bigl|\frac{a}{q}-\frac{a'}{q'}\Bigr|\le |\epsilon|+|\epsilon'|.
\label{eq:Separation}
\end{equation}
If $q'\ge N^{1/2}/2$ then $B'\ge q'\ge N^{1/2}/2$, and so
\[
\frac{x}{(B')^{1/2}}\ll \frac{x}{N^{1/4}}\le N^{3/4}.
\]
If instead $q'<N^{1/2}/2$ then $|\epsilon|\le 1/(qN^{1/2})\le 1/(2qq')$, so \eqref{eq:Separation} gives $|\epsilon'|\ge 1/(2qq')$. Hence, recalling that $q\le N^{1/2}$,
\[
B'\ge q'x|\epsilon'|\ge \frac{x}{2q}\ge \frac{x}{2N^{1/2}},
\]
and so
\[
\frac{x}{(B')^{1/2}}\ll x^{1/2}N^{1/4}\le N^{3/4}.
\]
In both cases \eqref{eq:Consistency} holds. This establishes \eqref{eq:BlockTarget}, and hence completes the proof of Theorem \ref{thrm:MainThm}.
\end{proof}

\end{document}